\documentclass[12pt]{amsart}
\usepackage{txfonts}
\usepackage{mathrsfs}
\usepackage{graphicx}

\usepackage{amsfonts}
\usepackage{amsthm}
\usepackage{graphics}
\usepackage{indentfirst}
\usepackage{cite}
\usepackage{latexsym}
\usepackage{amsmath}
\usepackage{amssymb}
\usepackage[dvips]{epsfig}
\usepackage{amscd}

\newtheorem{theorem}{Theorem}[section]
\newtheorem{remark}{Remark}[section]

\newtheorem{lemma}[theorem]{Lemma}

\newcommand{\bt}{\begin{theorem}}
\newcommand{\bl}{\begin{lemma}}
\newcommand{\el}{\end{lemma}}
\newcommand{\et}{\end{theorem}}

\newcommand{\bn}{\begin{eqnarray}}
\newcommand{\en}{\end{eqnarray}}
\newcommand{\bnn}{\begin{eqnarray*}}
\newcommand{\enn}{\end{eqnarray*}}

\newcommand{\ba}{\begin{aligned}}
\newcommand{\ea}{\end{aligned}}
\newcommand{\be}{\begin{equation}}
\newcommand{\ee}{\end{equation}}

\def\norm[#1]#2{\|#2\|_{#1}}

\newcommand{\BE}{{\boldsymbol{E}}}
\newcommand{\BH}{{\boldsymbol{H}}}
\newcommand{\BJ}{{\boldsymbol{J}}}

\newcommand{\Bg}{{\boldsymbol{g}}}
\newcommand{\Bn}{{\boldsymbol{n}}}
\newcommand{\Bv}{{\boldsymbol{v}}}

\newcommand{\BP}{{\boldsymbol{\Phi}}}

\makeatletter      
\@addtoreset{equation}{section}
\makeatother       

\begin{document}

\title[Regularity Results for Maxwell Equations]
{Regularity results for the time-harmonic  Maxwell equations with impedance boundary condition }

\author{Peipei Lu}
\address{Department of Mathematics, Soochow University, SuZhou\\
215006, People's Republic of China}
\email{pplu@suda.edu.cn}
\date{}

\author{Yun Wang}
\address{Department of Mathematics, Soochow University, SuZhou\\
215006, People's Republic of China}
\email{ywang3@suda.edu.cn}
\date{}

\author{Xuejun Xu}
\address{LSEC, Institute of
Computational Mathematics and Scientific/Engineering Computing,
Academy of Mathematics and System Sciences, Chinese Academy of
Sciences, People's Republic of China and
School of Mathematical Sciences, Tongji University, Shanghai, P. R. China}
\email{xxj@lsec.cc.ac.cn}

\maketitle
\begin{abstract}
This paper considers the time-harmonic Maxwell equations with impedance boundary condition.
We present $H^2$-norm bound and other high-order norm bounds for strong solutions.
The $H^2$-estimate have been derived in [M. Dauge, M. Costabel and S. Nicaise, Tech. Rep. 10-09, IRMAR (2010)] for the case with homogeneous boundary condition. Unfortunately,
their method can not be applied to the inhomogeneous case. The main novelty of this paper is that we follow
the spirit of the $H^1$-estimate in [R. Hiptmair,  A. Moiola and I. Perugia, Math. Models Methods Appl. Sci., 21(2011), pp. 2263-2287] and modify the proof by applying two inequalities of Friedrichs' type to make the $H^1$-estimate move into $H^2$-estimate and $W^{m, p}$-estimate.
Finally, the dependence of the regularity estimates on the wave number is obtained, which will play an important  role in the convergence analysis of the numerical solutions for the time-harmonic Maxwell equations.
\end{abstract}




\section{Introduction}
In this paper, we consider the following time-harmonic Maxwell boundary value problem:
\begin{equation}\label{Maxwell}
\left\{ \begin{array}{l}
- i \omega\epsilon \BE - \nabla \times \BH = - \BJ /i\omega,\ \ \ \ \mbox{in}\ \Omega,\\[2mm]
-i \omega \mu \BH + \nabla \times \BE = 0,\ \ \ \ \  \mbox{in}\ \Omega,\\[2mm]
\BH \times \Bn - \lambda \BE_T = \Bg /i\omega, \ \ \ \ \mbox{on}\ \partial \Omega,
\end{array}
\right.
\end{equation}
where $\BE$ is the electricity field, $\BH$ is the magnetic field, $\BJ$ is related to a given current density with ${\rm div}~\BJ  =0$ in $\Omega$. $\omega>0$ is a fixed wave number and  the material coefficients $\epsilon,\mu ,\lambda\in \mathbb R$ are assumed to be constant with $\epsilon,\mu>0$ and $\lambda \not=0$. $i$ denotes the imaginary unit.  $\Bn$ denotes the unit outward normal
to $\partial \Omega$, and $\BE_T = ( \Bn \times \BE ) \times \Bn$ denotes the tangential component of the electric field
$\BE$. The boundary condition is the standard impedance boundary condition which requires ${\Bg} \cdot \Bn = 0$,
thus, ${\Bg}_T = {\Bg}$.
The above Maxwell equations are of considerable importance in physics and mathematics.

The Maxwell's operator is strongly indefinite for high wave number $\omega$, which brings difficulties both in
theoretical analysis and numerical simulation. Various finite element methods \cite{Nedelec80,Nedelec86,Hiptmair02,Brenner07,Zhong09} have been developed to solve the Maxwell's problem. However, the error analysis and the uniqueness of the numerical solution can only be derived under the restrictive constraint $\omega^2h\leq C$, where $h$ is the mesh size. This
constraint is not practical in real world especially  for the three-dimensional case  with  large $\omega$.
Recently, Feng and Wu \cite{FW2014}  proposed and analyzed an interior penalty discontinuous Galerkin (IPDG) method
for the problem (\ref{pde_original}) with the high wave number, which is uniquely solvable without any mesh constraint.
This is a big step in the discretization and  theoretical analysis  of the finite element method for the  time-harmonic Maxwell equation with high wave number.
It is well known that the dimension of approximation DG space is
much larger than the dimension of the corresponding conforming space. To address
this issue, two HDG methods were presented in \cite{hp-HDG,HDG-2} for the numerical solution of the
Maxwell problem. These HDG methods retain the advantages of the standard DG
methods and result in a significant reduced degrees of freedom.
The methods in \cite{FW2014,hp-HDG,HDG-2} belong to a class of absolutely stable  methods for the time-harmonic Maxwell equations, and the
numerical results show their advantages over the standard finite element method in \cite{Monk}, especially when the wave number is large.
Unfortunately, the error estimates of the above methods are not complete, the theoretical analysis of them are all based on the following
assumption of the $H^2$ regularity estimate for  the electric field $\BE$:
\begin{align}
\label{regularity-H2} \|\BE\|_{H^2(\Omega)} \leq C
(1+\omega) \left( \|\BJ \|_{L^2(\Omega)} + \|\Bg\|_{H^{\frac12}(\partial \Omega) } \right),
\end{align}
 where the constant $C$ is independent of $\omega$. In this paper our main goal is to get the wave-explicit $H^2$-estimate for $\BE$.

By expressing $\BH$ in terms of $\BE$,  the above problem \eqref{Maxwell} is transformed into the following equations in terms of only $\BE$:
\begin{subequations}
\label{pde_original}
\begin{align}
\label{pde_original_1}
\nabla \times (\mu^{-1}\nabla \times \BE) - \omega^2 \epsilon\BE &= {\BJ} \qquad \rm{in}\ \Omega,\\
\label{BC-PDE}
(\mu^{-1}\nabla \times \BE )\times \Bn - i  \omega\lambda \BE_T &={\Bg} \qquad \rm{on}\ \partial\Omega.
\end{align}
\end{subequations}
Introduce the energy space
\begin{equation}\nonumber
H_{imp}({\rm curl}; \Omega)= \left\{\Bv \in L^2(\Omega): \ \nabla \times \Bv \in L^2(\Omega), \ \ \Bv_T \in L^2(\partial \Omega)    \right\},
\end{equation}
existence and uniqueness of solutions in $H_{imp}({\rm curl}; \Omega)$ was proved in Theorem 4.17 of \cite{Monk}, through the variational formulation of the problem \eqref{pde_original}.

One next topic is the regualrity of the unique solution. When $\Bg=0$,  M. Dauge, M. Costabel and S. Nicaise\cite{Dauge} made much effort on this homogeneous case. They gave an innovative proof for the $H^1$-estimate of $\BE$ and $ \BH$. Regarding the high-order estimate, they found that the variational formulation for the electric field $\BE$ does not define an elliptic problem, nor the variational formulation for the magnetic field $\BH$. Inspired by \cite{Nedelec01}, they considered a coupled regularized formulation for the full electromagnetic field $(\BE, \BH)$, and derived the high-order estimate for $(\BE, \BH)$. However,  their method can not be applied to the nonhomogeneous case $\Bg \neq 0$. And the dependence of the estimate on the wave number $\omega$ is not explicit. On the other hand,  Hiptmair-Moiola-Perugia\cite{HM}
 established some wave-explicit $H^1$-estimate.  It is shown that for $C^2$ domain, under the assumptions made in Theorem  4.1 in \cite{HM}, the $H^1$ regularity of both $\BE$ and $\nabla \times \BE$ can be obtained as follows:
\begin{align}
\label{regularity-H1} \|\nabla \times \BE\|_{H^1(\Omega)} +\omega \|\BE\|_{H^1(\Omega)}\leq C
(1+\omega)(\|{\BJ}\|_{L^2(\Omega)}+ \|{\Bg}\|_{L^2(\partial \Omega)})+\|{\Bg}\|_{H^{\frac12}(\partial \Omega)}.
\end{align}



In this paper, we will give a proof of the wave-explicit $H^2$-estimate. We will deal with the nonhomogeneous case, for which $\Bg$ may not vanish.  Let us highlight our main theoretical results:
Let $\Omega \subset \mathbb{R}^3$ be a bounded $C^{m+1}$-domain and star-shaped with respect to $B_\gamma (x_0)$. In addition to the assumptions made on $\BJ, \Bg$ and on the material coefficients, we assume that
$\BJ \in H^{m-1}(\Omega)$ and $\Bg\in H^{m-\frac12}(\partial \Omega)$. Then there exists one constant $C$ independent of $\omega$, but depending on $\Omega$, $\lambda$, $\epsilon$, $\mu$, such that, if $\BE$ is the solution to \eqref{pde_original},
\begin{equation}\nonumber   \begin{array}{ll}
\|\nabla \times \BE\|_{H^m(\Omega)} + \omega \| \BE\|_{H^m(\Omega)}
& \leq C (1+ \omega^m) \left( \|\BJ \|_{L^2(\Omega)} + \|\Bg\|_{L^2(\partial \Omega)} \right)  + C \omega^{m-1} \|\Bg\|_{H^{\frac12}(\partial \Omega) }   \\[2mm]&\ \ \ \ \displaystyle + C \sum_{k=1}^{m-1}  \omega^{m - k -1} \left(   \|\BJ\|_{H^k(\Omega)}
+ \|\Bg\|_{H^{k+\frac12}(\partial \Omega)} \right).
\end{array}
\end{equation}
Especially when $m=2$, we have
\begin{equation}\label{wang2} \begin{array}{ll}
\omega \| \BE\|_{H^2(\Omega)}
 & \leq C (1+ \omega^2) \left( \|\BJ \|_{L^2(\Omega)} + \|\Bg\|_{L^2(\partial \Omega)} \right) + C \omega \|\Bg\|_{H^{\frac12}(\partial \Omega) }\\
&  \ \ \ \ \ \ + C\left(   \|\BJ\cdot \Bn\|_{H^{\frac12}(\partial \Omega)}
+ \|\Bg\|_{H^{\frac{3}{2}}(\partial \Omega)} \right).
\end{array}
\end{equation}
Besides, for the case ${\rm div}~\BJ \neq 0$, some similar regularity results can be found in Remark \ref{non-divergence} of this paper. Furthermore, we extend the estimates into $W^{m, p}$ space.

As said above, the main idea of \cite{Dauge} for $H^2$-estimate is  rewriting the Maxwell equations in the form of elliptic equations of $\BE$ and $\nabla \times \BE$. To our knowledge, the method can not be applied to the case with nonhomogeneous boundary condition directly. Our proof is in the same spirit of the $H^1$-estimate in \cite{Dauge, HM}.   To make the $H^1$-estimate move into $H^2$-estimate and $W^{m, p}$-estimate, we apply two inequalities of Friedrichs' type, which is the main novelty of our method. Compared to the $H^2$-estimate in \cite{Dauge}, our proof is much simpler and appliable for a wider range of cases.  

The remainder of this paper is organized as follows: In Section 2, we introduce some basic function spaces and two inequalities of Friedrichs' type. Section 3 is devoted to the regularity estimates of  $\BE$ and $\nabla \times \BE$ in
$H^m$-norm while in Section 4  we extend
the  regularity estimates to $W^{m,p}$-space.

\section{Preliminaries}
First, let us introduce some function spaces. Let $L^p(\Omega)$ denote the usual vector-valued $L^p$-space over $\Omega$, $1\leq p \leq \infty.$ Let
$$W^{m, p}(\Omega)= \{ \Bv \in L^{p}(\Omega):\ D^\alpha \Bv \in L^p(\Omega),\ |\alpha| \leq m\},\ \ \ m \in \mathbb{N}.$$
Define the spaces:
$$L^p({\rm div}; \Omega) = \{\Bv\in L^p(\Omega):\ {\rm div}~\Bv \in L^p(\Omega)   \},$$
$$L^p({\rm curl}; \Omega) = \{ \Bv \in L^p(\Omega):\ {\rm curl}~\Bv \in L^p(\Omega)  \},$$
$$W^{m, p}({\rm div}; \Omega) = \{ \Bv\in W^{m, p}(\Omega): \ {\rm div}~\Bv \in W^{m, p}(\Omega)\},$$
$$W^{m, p}({\rm curl}; \Omega) = \{ \Bv\in W^{m, p}(\Omega): \ {\rm curl}~\Bv \in W^{m, p}(\Omega) \},$$
where ${\rm curl}~\Bv = \nabla \times \Bv$, the vorticity of $\Bv$.
When $p=2$, let us denote
$$H^m(\Omega)= W^{m, 2}(\Omega),\ \ \ H^m({\rm div}; \Omega) = W^{m, 2}({\rm div}; \Omega),\ \ \ \ H^m({\rm curl}; \Omega) = W^{m, 2}({\rm curl}; \Omega).$$

For every function $\Bv\in L^p( {\rm div}; \Omega)$, we denote $\Bv\cdot \Bn$ the normal boundary value of $\Bv$ defined in $W^{-\frac1p, p}(\partial \Omega)$,
\begin{equation}\nonumber
\forall \ \varphi \in W^{1, p^{\prime}}(\Omega), \ \ \ <\Bv\cdot \Bn,\ \varphi >_{\partial \Omega}
= \int_{\Omega}  \Bv \cdot \nabla \varphi \, dx + \int_{\Omega} {\rm div }~\Bv \cdot \varphi \, dx.
\end{equation}
And it holds that
\begin{equation}\label{normal}
\|\Bv \cdot \Bn \|_{W^{-\frac1p, p}(\partial \Omega)} \leq C \left(  \|\Bv\|_{L^p(\Omega)} + \| {\rm div}~\Bv\|_{L^p(\Omega)}           \right).
\end{equation}

For every function $\Bv\in L^p( {\rm curl}; \Omega)$, we denote $\Bv\times \Bn$ the tangential boundary value of $\Bv$ defined in $W^{-\frac1p, p}(\partial \Omega)$,
\begin{equation}\nonumber
\forall \ \boldsymbol{\varphi} \in W^{1, p^{\prime}}(\Omega), \ \ \ <\Bv\times \Bn,\ \boldsymbol{\varphi} >_{\partial \Omega}
= \int_{\Omega}  \Bv \cdot {\rm curl }~ \boldsymbol{\varphi} \, dx -  \int_{\Omega} {\rm curl }~\Bv \cdot \boldsymbol{\varphi} \, dx.
\end{equation}
And it holds that
\begin{equation}\label{tangential}
\|\Bv \times \Bn \|_{W^{-\frac1p, p}(\partial \Omega)} \leq C \left(  \|\Bv\|_{L^p(\Omega)} + \| {\rm curl}~\Bv\|_{L^p(\Omega)}           \right).
\end{equation}

\vspace{5mm}
Next, we will list two theorems for further use. Both theorems are Friedrichs' inequalities for vector fields. The first inequality gives  the estimate of $\nabla \Bv$ by ${\rm div}~\Bv$, ${\rm curl}~\Bv$ and $\Bv \cdot \Bn$.  Define the space
\begin{equation} \nonumber
X^{m, p}(\Omega)= \left\{ \Bv \in L^p(\Omega):\ {\rm div}~\Bv\in W^{m-1, p}(\Omega), {\rm curl}~\Bv \in W^{m-1, p}(\Omega),
\Bv\cdot \Bn \in W^{m-\frac1p, p}(\partial \Omega)              \right\}.
\end{equation}

\begin{theorem}\label{lemma1}Let $m \in \mathbb{N}$ and $\Omega$ be a bounded domain of class $C^{m+ 1}$. Then the space $X^{m, p}(\Omega)$ is continuously imbedded in $W^{m, p}(\Omega)$, and
for any $\Bv\in X^{m, p}(\Omega)$, we have the following estimate:
\begin{equation}\label{lemma1-1}
\|\Bv\|_{W^{m, p}(\Omega) }\leq C \left(  \|\Bv\|_{L^p(\Omega)} + \|{\rm curl}~\Bv\|_{W^{m-1, p}(\Omega)}+ \|{\rm div}~\Bv\|_{W^{m-1, p}(\Omega)}
+ \|\Bv\cdot \Bn \|_{W^{m- \frac1p, p}(\partial \Omega)}         \right),
\end{equation}
where $C$ depends on $\Omega, m, p$.
\end{theorem}

The second inequality gives the estimate of $\nabla \Bv$ by ${\rm div}~\Bv$, ${\rm curl}~\Bv$ and $\Bv \times \Bn$.
Define the space
\begin{equation}\nonumber
Y^{m, p}(\Omega) = \left\{ \Bv\in L^p(\Omega):\ {\rm div}~\Bv \in W^{m-1, p}(\Omega), \ {\rm curl}~\Bv \in W^{m-1, p}(\Omega), \
\Bv\times \Bn \in W^{m-\frac1p, p}(\partial \Omega)     \right\}.
\end{equation}

\begin{theorem}\label{lemma2}Let $m\in \mathbb{N}$ and $\Omega$ be a bounded domain of class $C^{m, 1}$. Then the space $Y^{m, p}(\Omega)$ is continuously imbedded in $W^{m, p}(\Omega)$, and
for any $\Bv\in Y^{m, p}(\Omega)$, we have the following estimate:
\begin{equation}\label{lemma2-1}
\|\Bv\|_{W^{m, p}(\Omega) }\leq C \left(  \|\Bv\|_{L^p(\Omega)} + \|{\rm curl}~\Bv\|_{W^{m-1, p}(\Omega)}+ \|{\rm div}~\Bv\|_{W^{m-1, p}(\Omega)}
+ \|\Bv\times \Bn \|_{W^{m- \frac1p, p}(\partial \Omega)}         \right),
\end{equation}
where $C$ depends on $\Omega, m, p$.
\end{theorem}


Both theorems have been proved in \cite{AS, KY}, so we omit the proof here.

\section{$H^m$-Estimates}

In this section, we will give the $H^m$-estimates for $\BE$ and $\nabla \times \BE$. Beforehand, we give the existence result for completeness, which can be found in \cite{Monk}.
\begin{theorem}
Let $\Omega$ be an open bounded domain, which either has a $C^2$-boundary or is a polyhedron. Suppose $\BJ \in L^2(\Omega)$ with ${\rm div}~\BJ = 0$ in $\Omega$,
and $\Bg\in L^2(\partial \Omega)$ with $\Bg \cdot \Bn = 0$ on $\partial \Omega$. Under the assumptions made on the material coefficients in Section 1, there exists one unique weak solution $\BE \in L^2({\rm curl} ; \Omega)$ to \eqref{pde_original}, satisfying ${\rm div}~\BE = 0$ in $\Omega$.
\end{theorem}

Now let us talk about the high-order regularity estimate of $\BE$. The first result deals with the $H^1$-estimate.

\begin{theorem}\label{theorem1}
Let $\Omega \subset \mathbb{R}^3$ be a bounded $C^2$-domain and star-shaped with respect to $B_{\gamma}(x_0)$. In addition to the assumptions made on $\BJ, \Bg$ and on the material coefficients in Section 1, we assume that
$\BJ \in L^2(\Omega)$ and
$\Bg\in H^{\frac12}(\partial \Omega)$. Then if $\BE$ is the solution to \eqref{pde_original},  $\BE$ belongs to the space $H^1({\bf curl}; \Omega),$ and there exists one constant $C$ independent of $\omega$, but depending on $\Omega$, $\lambda$, $\epsilon$, $\mu$, such that, 
\begin{equation}\label{3-1}
\|\mu^{-1} \nabla \times \BE\|_{H^1(\Omega)} + \omega \| \BE\|_{H^1(\Omega)} \leq C (1+ \omega) \left( \|\BJ \|_{L^2(\Omega)} + \|\Bg\|_{L^2(\partial \Omega)}\right) + C \|\Bg\|_{H^{\frac12}(\partial \Omega)}.
\end{equation}

\end{theorem}

\begin{remark}
Theorem \ref{theorem1} has been proved in \cite{HM}. Since the proof for high regularity estimates are in the same spirit as that for the $H^1$-estimate, we report the proof in detail. The proof here is slightly different from that in \cite{HM}, since we give a simpler proof for the estimate of $\nabla \times E$.
\end{remark}

The proof of Theorem \ref{theorem1} is based on the stability result derived by Hiptmair-Moiola-Perugia\cite{HM}. So first we list the stability result without proof.

\begin{theorem}\label{HM}
Let $\Omega \subset \mathbb{R}^3$ be a bounded $C^2$-domain which is star-shaped with respect to $B_{\gamma}(x_0)$. Suppose $\BJ \in L^2(\Omega)$ with ${\rm div}~\BJ = 0$ in $\Omega$, and  $\Bg \in L^2(\partial \Omega)$ with $\Bg\cdot \Bn =0$ on $\partial \Omega$. Under the assumptions made on  the material coefficients in Section 1,  there exist two positive constants $C$ independent of $\omega$, but depending on $d:= diam (\Omega)$, $\lambda, \epsilon$ and $\mu$, such that, if $\BE$ is the solution to \eqref{pde_original},
\begin{equation}\label{stability}
\|\mu^{-\frac12} \nabla \times \BE\|_{L^2(\Omega)} + \omega \|\epsilon^{\frac12} \BE\|_{L^2(\Omega)}
\leq C \left(\|\BJ \|_{L^2(\Omega)} + \|\Bg\|_{L^2(\partial \Omega)}\right).
\end{equation}
\end{theorem}

\vspace{2mm}

\begin{proof}[\bf Proof of Theorem \ref{theorem1}]
First, let us decompose $\BE$ as $$\BE= \BP^0 + \nabla \psi,$$
where $\BP^0$ satisfies
\begin{equation}\label{3-2} \left\{
\begin{array}{l}
{\rm div}~ \BP^0 = 0, \ \ \ \ \ \ \ \ {\rm in} \ \ \Omega,\\[2mm]
{\rm curl }~ \BP^0 = {\rm curl}~\BE,\ \ \ \ \ {\rm in} \ \ \Omega,\\[2mm]
\BP^0 \cdot \Bn = 0,  \ \ \ \ {\rm on} \ \ \partial \Omega,
\end{array} \right. \end{equation}
and $\psi$ satisfies
\begin{equation}\label{3-3} \left\{
\begin{array}{l}
\Delta~ \psi = 0, \ \ \ \ \  {\rm in} \ \ \Omega, \\[2mm]
\displaystyle \frac{\partial \psi}{\partial \Bn } =  \BE\cdot \Bn,\ \ \ \  {\rm on} \ \ \partial \Omega, \\[3mm]
\displaystyle \int_{\Omega} \psi\,  dx = 0.
\end{array}
\right.
\end{equation}
The above decomposition is a classical Helmholtz-Weyl decomposition. According to the orthogonality of the Helmholtz-Weyl decomposition, it holds that
\begin{equation}\nonumber
\left\| \BP^0 \right\|_{L^2(\Omega)} \leq \|\BE\|_{L^2(\Omega)},\ \ \ \mbox{and}\ \ \ \|\nabla \psi \|_{L^2(\Omega)} \leq \|\BE\|_{L^2(\Omega)}.
\end{equation}
By virtue of  Theorem \ref{lemma1},
\begin{equation} \label{3-4}
\left\|\BP^0 \right\|_{H^1(\Omega)} \leq
C \left( \|\nabla \times \BE\|_{L^2(\Omega)} + \left\|\BP^0 \right\|_{L^2(\Omega)} \right) \leq  C\left(  \| \nabla \times \BE\|_{L^2(\Omega)} + \|\BE\|_{L^2(\Omega)} \right).
\end{equation}
On the other hand, using  Poincar\'e's inequality, we  get  that
\begin{equation}\label{3-5}
\| \psi \|_{H^1 (\Omega) } \leq C \|\nabla \psi \|_{L^2(\Omega)} \leq C \|\BE\|_{L^2 (\Omega)}.
\end{equation}

Next, we will improve the regularity of $\psi$. The boundary condition \eqref{BC-PDE} can be rewritten as
\begin{equation}\label{3-6}
i \omega \lambda \nabla_{T} \psi = (\mu^{-1} \nabla \times \BE ) \times \Bn - i \omega \lambda \BP^0_{T} - \Bg,
\end{equation}
where $\nabla_{T}\psi$ is the tangential gradient of $\psi$, i.e., $\nabla_{T} \psi = (\Bn \times \nabla \psi) \times \Bn$.
According to Formula (3.52) in \cite{Monk},
\begin{equation} \label{3-7}
{\rm div}_T~\left[ (\mu^{-1}\nabla \times \BE ) \times {\Bn }  \right]
= - \Bn \cdot \left[ \nabla \times (\mu^{-1} \nabla \times \BE)    \right] =  - \Bn \cdot ( \BJ + \omega^2 \epsilon \BE),
\end{equation}
where ${\rm div}_T$ is the tangential divergence.
Hence it follows from the inequality \eqref{normal} that
\begin{equation}\label{3-8}
\begin{array}{l}
 \left\| {\rm div}_{T}~ [ ( \mu^{-1} \nabla \times \BE ) \times \Bn ]   \right\|_{H^{-\frac12}(\partial \Omega)} \\[3mm]
 \leq C \left[ \left\|\BJ + \omega^2 \epsilon \BE \right\|_{L^2(\Omega)} + \left\|{\rm div}~( \BJ +
\omega^2 \epsilon \BE) \right\|_{L^2(\Omega)}                      \right]\\[3mm]
\leq C \|\BJ\|_{L^2(\Omega)} + C  \omega^2 \|\BE\|_{L^2(\Omega)}.
\end{array}
\end{equation}
Meanwhile, by virtue of \eqref{3-4},
\begin{equation}\label{3-9}
\left \| i \omega \lambda \BP^0_T \right\|_{H^{\frac12}(\partial \Omega)}
\leq C \omega \| \BP^0\|_{H^1(\Omega)} \leq C \omega \left( \|\nabla \times \BE\|_{L^2(\Omega)} + \|\BE\|_{L^2(\Omega)}  \right).
\end{equation}
Combining the estimates \eqref{stability} and \eqref{3-8}-\eqref{3-9}, we have
\begin{equation}\label{3-9-add}
\begin{array}{l}
\omega \left \| {\rm div}_T \nabla_T \psi \right \|_{H^{-\frac12}(\partial \Omega)}\\
[3mm]
\leq  \left \|  {\rm div}_T \left[( \mu^{-1} \nabla \times \BE) \times \Bn \right]\right\|_{H^{-\frac12}(\partial \Omega)} + \|i \omega \lambda \BP^0_T\|_{H^{\frac12}(\partial \Omega)} +  \|\Bg\|_{H^{\frac12}(\partial \Omega)}   \\[3mm]
 \leq  C\| \BJ \|_{L^2(\Omega)}
+ C \omega^2 \|\BE\|_{L^2(\Omega)} + C \omega \left( \| \nabla \times \BE\|_{L^2(\Omega)} + \|\BE\|_{L^2(\Omega) }\right) + \|\Bg\|_{H^{\frac12}(\partial \Omega)} \\[2mm]
\leq C (1+ \omega) \left( \|\BJ\|_{L^2(\Omega)} + \|\Bg\|_{L^2(\partial \Omega)} \right)+ \|\Bg\|_{H^{\frac12}(\partial \Omega)}. 
\end{array}
\end{equation}

Hence, applying the elliptic lifting theorem for the Laplace-Beltrami operator on smooth surfaces, we have
\begin{equation}\label{3-10}
\begin{array}{l}
\omega \| \psi \|_{H^{\frac32}(\partial \Omega)}\\[2mm]
  \leq \displaystyle  C\omega  \left\| {\rm div}_T \nabla_T \psi \right\|_{H^{-\frac12}(\partial \Omega)} +
C \omega \left| \int_{\partial \Omega} \psi \ dS  \right|\\[3mm]
 \leq C \omega \left\| {\rm div}_T~  \nabla_T \psi  \right\|_{H^{-\frac12}(\partial \Omega)} + C \omega \|\psi \|_{H^1(\Omega)}\\[3mm]
\leq C (1+ \omega) \left( \|\BJ \|_{L^2(\Omega)}  +  \|\Bg\|_{L^2(\partial \Omega)}\right) + C \|\Bg\|_{H^{\frac12}(\partial \Omega)},
\end{array}
\end{equation}
where the last inequality is due to \eqref{stability}, \eqref{3-5} and \eqref{3-9-add}.
Then according to the regularity theory for the Dirichlet problem of Laplace equation,
\begin{equation}\label{3-11}
\omega \|\psi\|_{H^2(\Omega)} \leq C  (1 + \omega) \left(\|\BJ \|_{L^2(\Omega)}  + \|\Bg\|_{L^2(\partial \Omega)} \right) + C \|\Bg\|_{H^{\frac12}(\partial \Omega)},
\end{equation}
and consequently,
\begin{equation}\label{3-12}
\omega \|\BE\|_{H^1(\Omega)} \leq C (1+\omega)  \left(\|\BJ \|_{L^2(\Omega)}  + \|\Bg\|_{L^2(\partial \Omega)} \right)+ C \|\Bg\|_{H^{\frac12}(\partial \Omega)}.
\end{equation}

Now  we give the $H^1$-estimate for $\nabla \times \BE$. Since
\begin{equation}\label{3-15}
\left\{ \begin{array}{l}
{\rm div}~( \mu^{-1} \nabla \times \BE) = 0,\ \ \  {\rm in}\ \Omega,\\[2mm]
\nabla \times ( \mu^{-1} \nabla \times \BE) = \BJ + \omega^2 \epsilon \BE,\ \ \ {\rm in} \ \Omega, \\[2mm]
(\mu^{-1} \nabla \times \BE) \times \Bn =  i \omega \lambda \BE_T + \Bg,\ \  \ {\rm on} \ \partial \Omega.
\end{array}
\right.
\end{equation}
It follows from Theorem \ref{lemma2} that $\mu^{-1}\nabla \times \BE \in H^1(\Omega)$ and
\begin{equation}\label{3-16} \begin{array}{l}
 \| \mu^{-1} \nabla \times \BE\|_{H^1(\Omega)}\\[2mm]
 \leq C \|\BJ + \omega^2 \epsilon \BE\|_{L^2(\Omega)} + \| i \omega \lambda \BE_T + \Bg\|_{H^{\frac12}(\partial \Omega)} + C \|\mu^{-1} \nabla \times \BE\|_{L^2(\Omega)}\\[2mm]
\leq C \|\BJ \|_{L^2(\Omega)} + C \omega^2 \|\BE\|_{L^2(\Omega)} + C \|\Bg\|_{H^{\frac12}(\partial \Omega)} + C(1+\omega) \|\BE\|_{H^1(\Omega)}\\[2mm]
 \leq C(1+\omega)  \left(\|\BJ \|_{L^2(\Omega)}  + \|\Bg\|_{L^2(\partial \Omega)} \right)+ C \|\Bg\|_{H^{\frac12}(\partial \Omega)} ,
\end{array}
\end{equation}
where the last inequality is due to \eqref{stability} and  \eqref{3-12}.
That ends the proof of Theorem \ref{theorem1}.

\end{proof}

\begin{theorem}\label{theorem2}
Let $\Omega \subset \mathbb{R}^3$ be a bounded $C^{m+ 1}$-domain and star-shaped with respect to $B_\gamma (x_0)$. In addition to the assumptions made on $\BJ, \Bg$ and on the material coefficients in Sec. 1, we assume that
$\BJ \in H^{m-1}(\Omega)$ and $\Bg\in H^{m-\frac12}(\partial \Omega)$. Then  if $\BE$ is the solution to \eqref{pde_original},  $\BE$ belongs to the space $H^m({\rm curl}; \Omega)$ and there exists one constant $C$ independent of $\omega$, but depending on $\Omega$, $\lambda$, $\epsilon$, $\mu$, such that,
\begin{equation}\label{3-26} \begin{array}{l}
\|\nabla \times \BE\|_{H^m(\Omega)} + \omega \| \BE\|_{H^m(\Omega)} \\[2mm]
 \leq C (1+ \omega^m) \left( \|\BJ \|_{L^2(\Omega)} + \|\Bg\|_{L^2(\partial \Omega)} \right) + C \omega^{m-1} \|\Bg\|_{H^{\frac12}(\partial \Omega) }\\[2mm]\ \  \ \ \ \ \  + C \displaystyle \sum_{k=1}^{m-1}  \omega^{m - k -1} \left(   \|\BJ\|_{H^k(\Omega)}
+ \|\Bg\|_{H^{k+\frac12}(\partial \Omega)} \right).
\end{array}
\end{equation}
 In particular, when $m=2$,
\begin{equation} \label{3-26-1} \begin{array}{ll}
\omega \|\BE\|_{H^2(\Omega)}& \leq C (1 + \omega^2) \left( \|\BJ\|_{L^2(\Omega)} + \|\Bg\|_{L^2(\partial \Omega)} \right) + C\omega \|\Bg\|_{H^{\frac12}(\partial \Omega)}\\[2mm]& \ \ \ \ \ \ \ \  + C \left(  \|\BJ \cdot \Bn \|_{H^{\frac12}(\partial \Omega)} + \|\Bg\|_{H^{\frac32}(\partial \Omega)}  \right).
\end{array}
\end{equation}

\end{theorem}

\begin{proof} To simplify the discussion, we will write the proof for $m=2$ and the proof is similar when $m\geq 3$.
First, we use the same decomposition for $\BE$ as before, i. e. $\BE = \BP^0 + \nabla \psi$. According to Theorem \ref{theorem1}, the solution $E\in H^1({\rm curl}; \ \Omega)$. Applying Theorem \ref{lemma1} to the system \eqref{3-2}, we deduce that
\begin{equation}\label{3-27} \begin{array}{ll}
\|\BP^0 \|_{H^2(\Omega)} &\leq C \left( \|\nabla \times \BE\|_{H^1(\Omega)} + \|\BP^0\|_{L^2(\Omega)}  \right) \\[2mm]&\leq C \left(  \|\nabla \times \BE\|_{H^1(\Omega)} + \|\BE\|_{L^2(\Omega)} \right).
\end{array} \end{equation}
On the other hand, it follows from the classical regularity estimate for Laplace equation with Neumann boundary condition and the trace theorem that
\begin{equation}\label{3-28}
\|\psi\|_{H^2(\Omega)} \leq C \|\BE\|_{H^{\frac12}(\partial \Omega)} \leq C \|\BE\|_{H^1(\Omega)}.
\end{equation}

Next, we will improve the regularity of $\psi$. In this case,
\begin{equation}\label{3-29} \begin{array}{ll}
\left\|{\rm div}_T~ \left[  (\mu^{-1} \nabla \times \BE) \times \Bn \right] \right\|_{H^{\frac12}(\partial \Omega)}
& = \left\|  \Bn\cdot (\BJ + \omega^2 \epsilon \BE) \right\|_{H^{\frac12}(\partial \Omega)} \\[3mm]
& \leq \|\BJ \cdot \Bn \|_{H^{\frac12}(\partial \Omega)} + C \omega^2 \|\BE\|_{H^1(\Omega)},
\end{array}
\end{equation}
and
\begin{equation}\label{3-30}
\| i \omega \lambda \BP^0_T \|_{H^{\frac32} (\partial \Omega)}
\leq C \omega  \|\BP^0 \|_{H^2 (\Omega)} \leq C\omega  \left( \|\nabla \times \BE\|_{H^1(\Omega)}  + \|\BE\|_{L^2(\Omega)}\right).
\end{equation}
Taking the estimates \eqref{3-1}, \eqref{3-27}-\eqref{3-30} into the equality \eqref{3-6}, we have
\begin{equation} \begin{array}{l}
\omega \|{\rm div}_T \nabla_T \psi \|_{H^{\frac12}(\partial \Omega)} \\[2mm]
\leq C \|\BJ \cdot \Bn \|_{H^{\frac12}(\partial \Omega)} + C \omega^2 \|\BE\|_{H^1(\Omega)} + C \omega \left( \|\nabla \times \BE\|_{H^1(\Omega)}
+ \|\BE\|_{L^2(\Omega)} \right) + C \|\Bg\|_{H^{\frac32}(\partial \Omega)}\\[2mm]
\leq C \|\BJ \cdot \Bn \|_{H^{\frac12}(\partial \Omega)} + C (1+ \omega^2 ) \left(  \|\BJ\|_{L^2(\Omega)} + \|\Bg\|_{L^2(\partial \Omega)} \right) + C \omega \|\Bg\|_{H^{\frac12}(\partial \Omega)}+ C \|\Bg\|_{H^{\frac32}(\partial \Omega)}.
\end{array}
\end{equation}
Hence, it follows from the elliptic lifting theorem for the Laplace-Beltrami operator on smooth surfaces,
\begin{equation}\label{3-31} \begin{array}{l}
\omega \|\psi\|_{H^{\frac52} (\partial \Omega)} \\[3mm]
 \leq C \omega\left\| {\rm div}_T~ \nabla_T \psi \right\|_{H^{\frac12}(\partial \Omega)}
+ C \omega \left|    \int_{\partial \Omega} \psi\, dS     \right| \\[2mm]
\leq C \omega \|{\rm div}_T \nabla_T \psi \|_{H^{\frac12}(\partial \Omega)} + C \omega \|\psi \|_{H^2(\Omega)}\\[2mm]
\leq C (1 + \omega^2) \left( \| \BJ \|_{L^2(\Omega)} + \|\Bg\|_{L^2(\partial \Omega)}\right)
+ C \omega \|\Bg\|_{H^{\frac12} (\partial \Omega)} \\[2mm]\ \ \ \ \ \
 + C \left( \|\BJ \cdot \Bn  \|_{H^{\frac12}(\partial \Omega)} + \|\Bg\|_{H^{\frac32}(\partial \Omega)}      \right),
\end{array}
\end{equation}
which gives that
\begin{equation}\nonumber
\omega \|\psi \|_{H^3(\Omega)} \leq C (1 + \omega^2) \left( \| \BJ \|_{L^2(\Omega)} + \|\Bg\|_{L^2(\partial \Omega)}\right)
+ C \omega \|\Bg\|_{H^{\frac12} (\partial \Omega)} + C \left( \|\BJ \cdot \Bn \|_{H^{\frac12}(\partial \Omega)} + \|\Bg\|_{H^{\frac32}(\partial \Omega)}      \right) ,
\end{equation}
and consequently,
\begin{equation}\nonumber
\omega \|\BE\|_{H^2(\Omega)} \leq  C(1 + \omega^2) \left( \| \BJ \|_{L^2(\Omega)} + \|\Bg\|_{L^2(\partial \Omega)}\right)
+ C \omega \|\Bg\|_{H^{\frac12} (\partial \Omega)} + C \left( \|\BJ \cdot \Bn \|_{H^{\frac12}(\partial \Omega)} + \|\Bg\|_{H^{\frac32}(\partial \Omega)}      \right) .
\end{equation}

The $H^2$-estimate for $\nabla \times E$, as above,  follows from Theorem \ref{lemma2}.


\end{proof}

\begin{remark} If the assumption that $\Omega$ is star-shaped with respect to $B_\gamma (x_0)$ does not hold, i.e., we only assume that $\Omega$ is a bounded $C^{m+1}$-domain, we can also prove that $\BE$ belongs to $H^m({\rm curl}; \Omega)$ by following the same line as above. The regularity estimates can be written as follows,
\begin{equation}\nonumber
\|\nabla \times \BE\|_{H^m(\Omega)} + \|\BE\|_{H^m(\Omega)}
\leq C \|\BJ \|_{H^{m-1}(\Omega)} + C \|\Bg\|_{H^{m -  \frac12}(\partial \Omega)},
\end{equation}
where $C$ depends on $\omega$.
\end{remark}

\begin{remark}\label{non-divergence}
Let us give some discussion for the case ${\rm div}~\BJ \neq 0$. Assume $\BJ \in H^{m -1}(\Omega)$, $\Bg\in H^{m -\frac12}(\partial \Omega)$, but ${\rm div}~\BJ \neq 0$. There exists  a classical
Helmholtz-Weyl decomposition for $\BJ$, i. e. $\BJ = \BJ_0 + \nabla p$. Here $\BJ_0 \in H^{m-1}(\Omega)$ with ${\rm div}~\BJ_0 = 0$ in $\Omega$, and $p \in H^m (\Omega)$ satisfies
\begin{equation} \nonumber
\left\{
\begin{array}{l}
\Delta p = {\rm div}~\BJ, \ \ \ \ {\rm in} \ \Omega, \\[2mm]
p=0,\ \ \ \ \   {\rm on}  \  \partial \Omega.
\end{array}\right.
\end{equation}
Suppose $\BE_0$ is the solution to \eqref{pde_original} with $\BJ$ replaced by $\BJ_0$, it follows from Theorems \ref{theorem1} and \ref{theorem2} that
$\BE_0 \in H^m({\rm curl}, \Omega)$. Let $q = - p /(\omega^2 \epsilon)$. It is easy to check that $\BE = \BE_0 + \nabla q$ is the unique solution to \eqref{pde_original}, and it holds that
\begin{equation}\nonumber \begin{array}{l}
\omega \|\BE_0\|_{H^m(\Omega)} + \|\nabla \times \BE_0\|_{H^m(\Omega)} \\[2mm]
\leq C (1+ \omega^m) \left( \|\BJ \|_{L^2(\Omega)} + \|\Bg\|_{L^2(\partial \Omega)} \right) + C \omega^{m-1} \|\Bg\|_{H^{\frac12}(\partial \Omega) }\\[2mm] \ \ \ \ \ + C \sum_{k=1}^{m-1}  \omega^{m - k -1} \left(   \|\BJ\|_{H^k(\Omega)}
+ \|\Bg\|_{H^{k+\frac12}(\partial \Omega)} \right).
\end{array}
\end{equation}
and
\begin{equation}\nonumber
\omega^2 \|\nabla q\|_{H^{m-1}(\Omega)} \leq C \|\BJ \|_{H^{m-1}(\Omega)}.
\end{equation}
If we further assume that $\BJ \in H^{m-1}({\rm div}; \Omega)$, i.e., $\BJ \in H^{m-1}(\Omega)$ and ${\rm div}~\BJ \in H^{m-1}(\Omega)$, then
$\nabla q \in H^m(\Omega)$, and
\begin{equation}\nonumber
\omega^2 \|\nabla q \|_{H^m(\Omega)} \leq C \|{\rm div}~\BJ\|_{H^{m-1}(\Omega)}.
\end{equation}
\end{remark}

\section{ $W^{m, p}$-estimates}

In this section, we generalize the $H^m$-estimates for $\BE$ and $\nabla \times \BE$ to $W^{m, p}$-space.

\begin{theorem}\label{theorem3}
Let $\Omega \subset \mathbb{R}^3$ be a bounded $C^{m+ 1}$-domain and star-shaped with respect to $B_\gamma(x_0)$. In addition to the assumptions made on $\BJ, \Bg$ and on the material coefficients in Sec. 1, we assume that
$\BJ \in W^{m-1, p}(\Omega)$ and
$\Bg\in W^{m - \frac1p, p}(\partial \Omega)$, $p > 2$. Then if $\BE$ is the solution to \eqref{pde_original}, $\BE$ belongs to the space $W^{m, p}({\rm curl}; \Omega)$. Moreover, there exists one constant $C$ independent of $\omega$, but depending on $\Omega$, $\lambda$, $\epsilon$, $\mu$,  such that,  when $2< p \leq 6$,
\begin{equation}\label{4-00}
\|\nabla \times \BE\|_{W^{m, p}(\Omega)} + \omega \|\BE\|_{W^{m, p}(\Omega)} \leq C( 1 + \omega^{m+1}) \left(\|\BJ \|_{W^{m-1, p}(\Omega)}
+ \|\Bg\|_{W^{m - \frac1p, p}(\partial \Omega)}  \right),
\end{equation}
and when $p>6$,
\begin{equation}
\|\nabla \times \BE\|_{W^{m, p}(\Omega)} + \omega \|\BE\|_{W^{m, p}(\Omega)} \leq C ( 1 + \omega^{m+2}) \left(\|\BJ \|_{W^{m-1, p}(\Omega)}
+ \|\Bg\|_{W^{m - \frac1p, p}(\partial \Omega)}  \right).
\end{equation}

\end{theorem}

\begin{proof}
The proof is similar to that of Theorem \ref{theorem1}. We will write the proof for $m=1$. For some technical reasons, we divide the proof into two cases: $2< p \leq 6$ and $p >6$.

{\bf Case I:\ $2< p \leq 6$}\ \ Since $\BJ \in L^p(\Omega) \subset L^2(\Omega)$, and $\Bg\in W^{1- \frac1p, p}(\partial \Omega) \subset H^{\frac12}(\partial \Omega)$, it follows from Theorem \ref{theorem1} that $\BE \in H^1({\rm curl}; \Omega)$, and
\begin{equation}\label{4-0} \begin{array}{ll}
\omega \|\BE\|_{H^1(\Omega)} + \|\nabla \times \BE\|_{H^1(\Omega)} & \leq  C(1+ \omega) \left( \|\BJ \|_{L^2(\Omega)} +
\|\Bg\|_{L^2(\partial \Omega) }\right) + C \|\Bg\|_{H^{\frac12}(\partial \Omega)}\\[2mm]
& \leq C (1 + \omega) \left( \|\BJ \|_{L^p(\Omega)} + \|\Bg\|_{W^{1-\frac1p, p}(\partial \Omega) }\right).
\end{array}
\end{equation}

 We use the same Helmholtz-Weyl decomposition for $\BE$ as before, i. e., $\BE = \BP^0 + \nabla \psi$. According to classical theory for Helmholtz-Weyl decomposition\cite{AS, KY}, it holds that
\begin{equation}\nonumber
\left\| \BP^0 \right\|_{L^p(\Omega)} + \|\nabla \psi\|_{L^p(\Omega)} \leq C \|\BE\|_{L^p(\Omega)}.
\end{equation}
Hence, it follows from Theorem \ref{lemma1}  that
\begin{equation}\label{4-1} \begin{array}{ll}
\|\BP^0\|_{W^{1, p}(\Omega)}& \leq  C \left(  \|\nabla \times \BE\|_{L^p(\Omega)} + \|\BP^0\|_{L^p(\Omega)}     \right)\\[2mm]
& \leq C\left(\|\nabla \times \BE\|_{L^p(\Omega)} + \| \BE\|_{L^p(\Omega)} \right) \\[2mm]
& \leq C \left( \|\nabla \times \BE\|_{H^1(\Omega)} + \|\BE\|_{H^1(\Omega)} \right),
\end{array}
\end{equation}
where for the last inequality we used the Sobolev embedding result $L^p(\Omega) \subset H^1(\Omega)$, for $2< p \leq 6$(this is the technical reason why we divide the proof into two cases).

Next we will give the estimates for $\psi$. If follows from the inequality \eqref{normal} and the estimate \eqref{4-0} that
\begin{equation}\label{4-4} \begin{array}{l}
\left \|{\rm div}_T~[(\mu^{-1} \nabla \times \BE) \times \Bn] \right\|_{W^{-\frac1p, p}(\partial \Omega)}\\[2mm]
= \left \| \Bn \cdot ( \BJ + \omega^2 \epsilon \BE ) \right\|_{W^{-\frac1p, p}(\partial \Omega)}\\[2mm]
\leq C \left\|\BJ + \omega^2 \epsilon \BE  \right\|_{L^p(\Omega)} + C \left\|{\rm div} ~(\BJ + \omega^2 \epsilon \BE)      \right \|_{L^p(\Omega)}\\[2mm]
\leq C \|\BJ \|_{L^p(\Omega)} + C \omega^2 \|\BE\|_{L^p(\Omega) }\\[2mm]
\leq C \|\BJ \|_{L^p(\Omega)} + C \omega^2 \|\BE\|_{H^1(\Omega)},
\end{array}
\end{equation}
On the other hand,
\begin{equation}\label{4-5}
\left\| i \omega \lambda \BP^0_T \right\|_{W^{1 - \frac1p, p}(\partial \Omega)}
\leq C \omega \left\|\BP^0 \right\|_{W^{1, p}(\Omega)} \leq C \omega \left( \|\nabla \times \BE\|_{H^1(\Omega)}  + \|\BE\|_{H^1(\Omega)} \right).
\end{equation}
Collecting the estimates \eqref{4-0}-\eqref{4-5}, we can get
\begin{equation}\nonumber \begin{array}{l}
\omega \|\psi\|_{W^{2- \frac1p, p}(\partial \Omega)} \\[2mm]
\leq C  \left\|{\rm div}_T~\omega \nabla_T \psi \right\|_{W^{-\frac1p, p}(\partial \Omega)} + C \left|  \int_{\partial \Omega} \omega \psi \, dS     \right|\\[3mm]
 \leq C \left\| {\rm div}_T ~ \left[  (\mu^{-1} \nabla \times \BE) \times \Bn      \right] \right\|_{W^{-\frac1p, p}(\partial \Omega)} + C \left \|i \omega
\lambda \BP^0_T \right \|_{W^{1-\frac1p, p}(\partial \Omega)}\\[3mm] \ \ \ \ + \|\Bg\|_{W^{1-\frac1p, p}(\partial \Omega)} + C\omega \|\psi \|_{H^1(\Omega)}\\[3mm]
 \leq C \| \BJ \|_{L^p(\Omega)} + C \|\Bg\|_{W^{1- \frac1p, p}(\partial \Omega)} + C (1+\omega)\omega \|\BE\|_{H^1(\Omega)} + C \omega \|\nabla \times \BE\|_{H^1(\Omega)}\\[2mm]
\leq C (1+ \omega^2) \left( \|\BJ \|_{L^p(\Omega)} + \|\Bg\|_{W^{1-\frac1p, p}(\partial \Omega)}     \right).
\end{array}
\end{equation}
By the regularity theory for Laplace equations, we have
\begin{equation}\label{4-7}
\omega \|\psi\|_{W^{2, p}(\Omega)}\leq C ( 1 + \omega^2) \left(\|\BJ \|_{L^p(\Omega)} +  \|\Bg\|_{W^{1-\frac1p, p}(\partial \Omega)}\right),
\end{equation}
which together with \eqref{4-0}- \eqref{4-1} gives that
\begin{equation}\label{4-8}
\omega \|\BE\|_{W^{1, p}(\Omega)}\leq C (1+ \omega^2) \left( \|\BJ \|_{L^p(\Omega)} +  \|\Bg\|_{W^{1-\frac1p, p}(\partial \Omega)}\right).
\end{equation}

Similarly, we give the $W^{1, p}$-estimate for $\nabla \times E$. Applying Theorem \ref{lemma2}, one can easily deduce that
\begin{equation}\label{4-10} \begin{array}{l}
\|\mu^{-1}\nabla \times \BE\|_{W^{1, p}(\Omega)} \\[2mm]
 \leq C \|J + \omega^2\epsilon \BE\|_{L^p(\Omega)} + C \| i \omega \lambda \BE_T + \Bg\|_{W^{1- \frac1p, p}(\partial \Omega)} + C \|\mu^{-1}\nabla \times \BE\|_{L^p(\Omega) }\\[2mm]
\leq C (1+ \omega^2) \left( \|\BJ \|_{L^p(\Omega)} +  \|\Bg\|_{W^{1-\frac1p, p}(\partial \Omega)}\right).
\end{array}
\end{equation}

{\bf Case II:  $p> 6$}\ \ As before, $\BE$ is
decomposed as
$\BE = \BP^0 + \nabla \psi.$
Since $\BJ \in L^p(\Omega) \subset L^6(\Omega)$ and $\Bg\in W^{1-\frac1p, p}(\partial \Omega)\subset W^{\frac56, 6}(\partial \Omega)$, the solution
$\BE \in W^{1, 6}({\rm curl}; \Omega)$. Moreover,
\begin{equation} \label{4-20} \begin{array}{l}
\omega\|\BE\|_{W^{1, 6}(\Omega)} + \| \nabla \times \BE\|_{W^{1, 6}(\Omega)} \\[2mm]
 \leq C (1+ \omega^2) \left( \|\BJ \|_{L^6 (\Omega)} +  C \|\Bg\|_{W^{\frac56, 6}(\partial \Omega)}\right) \\[2mm]
 \leq C (1+ \omega^2) \left( \|\BJ \|_{L^p(\Omega)} + C \|\Bg\|_{W^{1 -\frac1p, p}(\partial \Omega)}\right).
\end{array}
\end{equation}
According to Theorem \ref{lemma1} and Sobolev embedding theorem,
\begin{equation}\nonumber
\|\BP^0 \|_{W^{1, p}(\Omega)} \leq C \|\BP^0 \|_{W^{2, 6}(\Omega)} \leq C \left(  \|\nabla \times \BE\|_{W^{1, 6}(\Omega)} + \|\BE\|_{L^6(\Omega)} \right).
\end{equation}
Hence,
\begin{equation}\label{4-21}
\omega \|\BP^0\|_{W^{1, p}(\Omega)} \leq C (1 + \omega^3) \left( \|\BJ \|_{L^p(\Omega)} + \|\Bg\|_{W^{1 -\frac1p, p}(\partial \Omega)}  \right).
\end{equation}
Meanwhile,
\begin{equation}\label{4-22}
\omega \|\psi \|_{H^1(\Omega)} \leq C \omega \|\BE\|_{L^2(\Omega)} \leq C \|\BJ \|_{L^2(\Omega)} + C \|\Bg\|_{L^2(\partial \Omega)}
\leq C \|\BJ \|_{L^p (\Omega)} + C \|\Bg\|_{W^{1- \frac1p, p}(\partial \Omega)}.
\end{equation}

Next we will get the $W^{1, p}$-estimate for $\nabla \psi$.
\begin{equation}\label{4-23}
\begin{array}{l}
\omega \|\psi\|_{W^{2- \frac1p, p}(\partial \Omega)} \\[2mm] \displaystyle
\leq C \left \| {\rm div}_T~\omega \nabla_T \psi \right \|_{W^{-\frac1p, p}(\partial \Omega)} + C \left|  \int_{\partial \Omega} \omega \psi \, dS     \right|\\[3mm]
 \leq C \left\| {\rm div}_T ~ \left[  (\mu^{-1} \nabla \times \BE) \times \Bn      \right] \right\|_{W^{-\frac1p, p}(\partial \Omega)} + C \omega  \left\|
\BP^0_T \right\|_{W^{1-\frac1p, p}(\partial \Omega)} \\[3mm]\ \ \ \ + C \|\Bg\|_{W^{1-\frac1p, p}(\partial \Omega)} + C \omega\|\psi \|_{H^1(\Omega)}\\[3mm]
\leq C \left\| \BJ + \omega^2 \epsilon \BE  \right\|_{L^p(\Omega)} + C \omega \| \BP^0\|_{W^{1, p}(\Omega)} + C \|\Bg\|_{W^{1-\frac1p, p}(\partial \Omega)}
+ C \omega\|\psi \|_{H^1(\Omega)}.
\end{array}
\end{equation}
Combining the estimates \eqref{4-20}-\eqref{4-23}, we have
\begin{equation}\nonumber
\omega\|\psi \|_{W^{2, p} (\Omega)} \leq C (1 + \omega^3 ) \left( \|\BJ \|_{L^p(\Omega)} +  \|\Bg\|_{L^p(\partial \Omega)}\right)
+ C (1+ \omega^2 )\|\Bg\|_{W^{1-\frac1p, p}(\partial \Omega)},
\end{equation}
and consequently,
\begin{equation}\label{4-25}
\omega \|\BE\|_{W^{1, p}(\Omega)} \leq C (1 + \omega^3 ) \left( \|\BJ \|_{L^p(\Omega)} +  \|\Bg\|_{W^{1 -\frac1p, p} (\partial \Omega)}\right)  .
\end{equation}

The $W^{1, p}$-estimate for $\nabla \times \BE$ follows from Theorem \ref{lemma2}.
\end{proof}

When $1< p < 2$, the problem becomes a little more complicated. When $\BJ \in L^p(\Omega)$, $\Bg\in W^{1-\frac1p, p}(\partial \Omega)$, the existence of solutions is not clear. So we put aside this case. Instead, we assume more regularity on $\BJ$ and $\Bg$. When $\BJ \in W^{m-1, p}(\Omega)$, $\Bg\in W^{m- \frac1p, p}(\partial \Omega)$ and $m \geq 2$, according to Sobolev embedding theorem, $\BJ \in L^2(\Omega)$ and $\Bg\in L^2(\partial \Omega)$. Hence the existence of the solution $\BE\in L^2({\rm curl}; \Omega)$ is guaranteed. Following the same proof as above, we can prove the following result:
\begin{theorem}
Let $m \geq 2$ and  $\Omega \subset \mathbb{R}^3$ be a bounded $C^{m+ 1}$-domain. In addition to the assumptions made on $\BJ, \Bg$ and on the material coefficients in Sec. 1, we assume that $\BJ \in W^{m-1, p}(\Omega)$ and
$\Bg\in W^{m- \frac1p, p}(\partial \Omega)$. Then if $\BE$ is the solution to \eqref{pde_original},
$\BE$ belongs to the space $W^{m, p}(\bf {curl}; \Omega)$, and 
 there exists one constant $C$ independent of $\omega$, but depending on $\Omega$, $\lambda$, $\epsilon$, $\mu$,$p$,  such that, 
\begin{equation}\label{4-30}
\omega \|\BE\|_{W^{m, p}(\Omega)} + \|\nabla \times \BE\|_{W^{m, p}(\Omega)} \leq C(1 + \omega^{m+1})  \left( \|\BJ\|_{W^{m-1, p}(\Omega)} + \|\Bg\|_{W^{m-\frac1p, p}(\partial \Omega)} \right).
\end{equation}

\end{theorem}

\vspace{5mm}

{\bf Acknowledgment}\ \ The work of the second author was partially supported by NSFC grant No. 11671289. The work of the third author was partially supported by NSFC grant No. 11671302.

\end{document}